\documentclass[final]{my-elsarticle}

\usepackage{hyperref}
\usepackage{amsmath,amssymb,amsthm}

\newtheorem{theorem}{Theorem}
\newtheorem{corollary}[theorem]{Corollary}
\newtheorem{lemma}[theorem]{Lemma}
\newtheorem{definition}[theorem]{Definition}
\newtheorem{notation}[theorem]{Notation}
\newtheorem{remark}[theorem]{Remark}
\newtheorem{example}[theorem]{Example}

\allowdisplaybreaks 

\begin{document}
	
\begin{frontmatter}

\title{Weak Pontryagin's Maximum Principle for Optimal 
Control Problems Involving a General Analytic Kernel\tnoteref{note}}

\tnotetext[note]{This is a preprint whose final form is published by Elsevier in the book 
\emph{Fractional Order Systems and Applications in Engineering}.
Submitted 10/Sept/2020; Revised 22/Nov/2020 and 18/May/2021; Accepted 05/Sept/2021.
This research is part of first author's Ph.D. project, 
which is carried out at University of Aveiro.}

\author{Fa\"{\i}\c{c}al Nda\"{\i}rou}
\ead{faical@ua.pt}
\ead[url]{https://orcid.org/0000-0002-0119-6178}

\author{Delfim F. M. Torres}
\ead{delfim@ua.pt}
\ead[url]{https://orcid.org/0000-0001-8641-2505}

\address{Center for Research and Development in Mathematics and Applications (CIDMA),
Department of Mathematics, University of Aveiro, 3810-193 Aveiro, Portugal}


\begin{abstract}
We prove a duality relation and an integration by parts formula 
for fractional operators with a general analytical kernel. Based
on these basic results, we are able to prove a new Gr\"onwall's inequality
and continuity and differentiability of solutions of control differential equations. 
This allow us to obtain a weak version of Pontryagin's maximum principle.
Moreover, our approach also allow us to consider mixed problems with both integer 
and fractional order operators and derive necessary optimality conditions 
for isoperimetric variational problems and other problems of the calculus of variations.
\end{abstract}


\end{frontmatter}


\section{Introduction}

Integration by parts is a powerful tool when two functions are multiplied together, 
but is also helpful in many other ways. In fact, applications of integration by parts 
abound, including the laws of Bessel bridges via hypergeometric functions \cite{MR4125793},
surface measures on levels sets induced by Brownian functionals \cite{MR4092082},
reductions of Feynman integrals \cite{MR4089288}, and approximation theory \cite{MR4017113}.

Another important tool, allowing to bound a function that is known to satisfy 
a certain differential or integral inequality by the solution of the corresponding 
differential or integral equation, is Gr\"onwall's inequality \cite{MR2016992}. 
It provides useful estimates in ordinary and stochastic differential 
equations \cite{MR4061730}, stability analysis \cite{MR4133248}, 
and fractional difference \cite{MR4062051} and differential equations \cite{MR4020061}.

Along the years, conjugation of integration by parts with Gr\"onwall's inequality has shown
a myriad of interesting results in several different areas, e.g., 
in probability theory and stochastic processes \cite{MR4068313},
systems and control theory \cite{MR3436584}, 
and fractional optimal control \cite{MR3736166}.

Fractional optimal control and the fractional calculus of variations
are concerned with the analysis and derivations of necessary optimality conditions 
for optimization problems involving fractional operators \cite{MR3822307,MR4116679}. 
For smooth and unconstrained data, optimal control problems 
of Lagrange form can be seen as a generalization of the calculus of variations 
\cite{liberzon,MR3331286}. In fact, maximum principles or optimality conditions 
can be obtained from variational analysis approaches. In particular, 
the Pontryagin maximum principle \cite{MR0186436} takes then a special weaker 
form in which the Hamiltonian maximality condition is reduced to a null 
derivative of the Hamitonian function along the extremals \cite{faical}. 
Here we prove integration by parts and a Gr\"onwall's inequality 
for fractional operators with a general analytical kernel in the sense
of Fernandez, \"{O}zarslan and Baleanu \cite{arran},
and apply it in the context of the fractional calculus 
of variations and fractional optimal control. 

Recent advancements within the fractional calculus research community 
has included proposals for general classes of operators, covering many 
of the numerous diverse definitions of fractional integrals and derivatives 
under a single generalised operator \cite{MR4034743}. As already mentioned, 
here we deal with one such advancement: the Fernandez--\"{O}zarslan--Baleanu (FOB) 
fractional calculus of 2019 \cite{arran}. Although recent, the ideas of the FOB
fractional calculus have already find many interesting applications, for example
in the identification of space-dependent source terms in nonlocal problems
\cite{MR4163085}, tempered fractional calculus \cite{MR3995283}, and
on the determination of a source term for fractional Rayleigh--Stokes 
equations with random data \cite{MR4062054}. Here we introduce and
develop a FOB fractional optimal control theory.

The manuscript is organized as follows. In Section~\ref{sec:02},
we briefly review the FOB fractional calculus, recalling necessary
notions and results while fixing notation. Our original results begin
with Section~\ref{sec:03}, where we prove an important duality relation
(Lemma~\ref{lem:dual}), integration by parts (Lemma~\ref{lem:int:parts}),
and Gr\"onwall's inequality (Theorem~\ref{thm:GI}). Follows our main applications:
in Section~\ref{sectionR}, we prove continuity and differentiability 
of solutions to control differential equations (respectively Lemmas~\ref{cont}
and \ref{different}), from which we prove a weak version of Pontryagin's maximum principle
for FOB optimal control problems (Theorem~\ref{theo}). As corollaries,
we deduce Euler--Lagrange necessary optimality conditions for the fundamental
FOB fractional problem of the calculus of variations (Corollary~\ref{cor:EL:fund:CoV})
as well as isoperimetric problems (Corollary~\ref{cor:EL:iso}).


\section{Preliminaries}
\label{sec:02}

We begin by recalling the general analytic kernel fractional calculus 
of Fernandez, \"{O}zarslan and Baleanu, proposed in 2019 \cite{arran}.

\begin{definition}[See \cite{arran}]
\label{def1}
Let $[a, b]$ be a real interval, $\alpha$ be a real parameter in $[0, 1]$, 
$\beta $ be a complex parameter with non-negative real part, and $R$ be a 
positive number satisfying $R> (b-a)^{Re(\beta)}$. Let $A$ be a complex 
analytic function on the disc $D(0, R)$ and defined on this disc 
by the locally uniformly convergent power series
\[
A(x)= \sum^{\infty}_{n=0} a_n x^n.
\] 
The left and right-sided fractional integrals with general analytic kernel
of a function $x: [a, b]\rightarrow \mathbb{R}$ are defined by 
\[
^{A}I^{\alpha, \beta}_{a+} x(t):= \int^t_a(t-s)^{\alpha-1}A\left( (t-s)^{\beta}\right)x(s)ds 
\]
and 
\[
 ^{A}I^{\alpha, \beta}_{b-} x(t):= \int^b_t(s-t)^{\alpha-1}A\left( (s-t)^{\beta}\right)x(s)ds,
\]
respectively.
\end{definition}

\begin{remark}
Note that the Riemann--Liouville fractional integral 
is a special case of Definition~\ref{def1} given by 
\[
^{RL}I^{\alpha}_{a+}x(t)
= \frac{1}{\Gamma(\alpha)A(1)}{^{A}I^{\alpha, 0}_{a+}x(t)},
\]
for any arbitrary choice of the analytic function $A$.
\end{remark}
 
\begin{notation}
For any analytic function $A$ as in Definition \ref{def1}, 
we denote $A_{\Gamma}$ to be the transformation function
\[
A_{\Gamma}(x):= \sum^{\infty}_{n=0}a_n\Gamma(\beta n + \alpha)x^n.
\]
\end{notation}

\begin{lemma}[Series formula \cite{arran}] 
For any integrable function $x(t)$, $t \in [a, b]$, 
the following uniformly convergent series formulas for $^{A}I^{\alpha, \beta}_{a+} x$
and for $^{A}I^{\alpha, \beta}_{b-} x$, as functions on $[a,b]$, hold:
\[
^{A}I^{\alpha, \beta}_{a+} x(t)
:= \sum^{\infty}_{n=0} a_n\Gamma (\beta n + \alpha)^{RL}I^{\alpha + n\beta}_{a+}x(t)
\]
and 
\[
^{A}I^{\alpha, \beta}_{b-} x(t):= \sum^{\infty}_{n=0} 
a_n\Gamma (\beta n + \alpha)^{RL}I^{\alpha + n\beta}_{b-}x(t),
\]
respectively. Here $^{RL}I^{\alpha + n\beta}_{a+}$ and 
$^{RL}I^{\alpha + n\beta}_{b-}$ are the left and right-sided 
Riemann--Liouville fractional integrals 
of order $\alpha + n\beta $, respectively.
\end{lemma}

Let $a, b, \alpha, \beta$ and $A$ be as in Definition~\ref{def1}.

\begin{definition}
The left and right Rieman--Liouville fractional derivatives 
with general analytic kernels of a function $x: [a, b] \rightarrow \mathbb{R}$ 
with sufficient differentiability properties are defined by
\[
^{A}_{RL}D^{\alpha, \beta}_{a+} x(t)
= \frac{d}{dt}\Big( ^{\bar{A}}I^{1-\alpha, \beta}_{a+}x(t)\Big) 
\quad \text{ and } 
\quad ^{A}_{RL}D^{\alpha, \beta}_{b-} x(t)
= -\frac{d}{dt}\Big( ^{\bar{A}}I^{1-\alpha, \beta}_{b-}x(t)\Big),
\]
where the function $\bar{A}$ used on the right-hand side is an analytic function 
defined by $\bar{A}(x)=\sum_{n=0}^{\infty}\bar{a}_n x^n$ 
and such that $A_{\Gamma}\cdot \bar{A}_{\Gamma}=1$.
\end{definition}

\begin{definition}
The left and right Caputo fractional derivatives with general analytic kernels 
of a function $x: [a, b] \rightarrow \mathbb{R}$ with sufficient 
differentiability properties are defined by
\[
^{A}_{C}D^{\alpha, \beta}_{a+} x(t)= ^{\bar{A}}I^{1-\alpha, \beta}_{a+}x'(t) 
\quad \text{ and } \quad ^{A}_{C}D^{\alpha, \beta}_{b-} x(t)
= -^{\bar{A}}I^{1-\alpha, \beta}_{b-}x'(t),
\]
where the function $\bar{A}$ used on the right-hand side is an analytic 
function defined by $\bar{A}(x)=\sum_{n=0}^{\infty}\bar{a}_n x^n$ 
and such that $A_{\Gamma}\cdot \bar{A}_{\Gamma}=1$.
\end{definition}

\begin{remark}
Note that the classical integer order derivative is obtained, up to a multiplicative constant, 
when $\alpha = 1$ and $\beta =0$.
\end{remark}

\begin{theorem}[Semi group property \cite{arran}]
\label{semigroup} 
Let $a, b, A$ be as in Definition~\ref{def1}, and fix 
$\alpha_1, \alpha_2, \beta \in \mathbb{C}$ with non-negative real parts. 
The semigroup property 
\[
{^{A}I^{\alpha_1, \beta}_{a+}} \circ {^{A}I^{\alpha_2, \beta}_{a+}} x(t) 
= {^{A}I^{\alpha_1 + \alpha_2, \beta}_{a+}} x(t)
\]
is uniformly valid (regardless of $\alpha_1, \alpha_2, \beta $ and $f$) 
if, and only if, the following condition is satisfied for all non-negative integers $k$:
\[
\sum_{m+n=k}a_n(\alpha_1, \beta)a_m(\alpha_2, \beta)
\Gamma(\alpha_1 + n\beta)\Gamma(\alpha_2 + n\beta)
= a_k(\alpha_1 + \alpha_2, \beta)\Gamma(\alpha_1 + \alpha_2+ k\beta).
\]
\end{theorem}


\section{Fundamental Properties}
\label{sec:03}

We begin by proving some technical but important results.

\begin{lemma}[Duality operation] 
\label{lem:dual}
For any functions $x(t)$ and $y(t)$, $t \in [a, b]$, the following duality relation holds:
\[
\int^b_a x(t)^{A}I^{\alpha, \beta}_{a+} y(t) dt 
= \int^b_a y(t)^{A}I^{\alpha, \beta}_{b-} x(t)dt.
\]
\end{lemma}

\begin{proof}
By the series formula, we have that
\begin{equation}
\label{series}
\int^b_a x(t)^{A}I^{\alpha, \beta}_{a+} y(t) dt 
= \int^b_a x(t)\sum^{\infty}_{n=0}a_n 
\Gamma(\beta n + \alpha)^{RL}I^{\alpha + n\beta}_{a+} y(t)dt.
\end{equation}
Since the series in the right-hand side of \eqref{series} 
is uniformly convergent, it follows that
\[
\int^b_a x(t)^{A}I^{\alpha, \beta}_{a+} y(t) dt 
= \sum^{\infty}_{n=0}a_n \Gamma(\beta n + \alpha) 
\int^b_a x(t) ^{RL}I^{\alpha + n\beta}_{a+} y(t)dt
\]
and, by duality of Riemann--Liouville integral operators, one has
\[
\int^b_a x(t) ^{RL}I^{\alpha + n\beta}_{a+} y(t)dt 
=  \int^b_a y(t) ^{RL}I^{\alpha + n\beta}_{b-} x(t)dt
\]
for any $n \in \mathbb{N}$, which leads to 
\[
\sum^{\infty}_{n=0}a_n \Gamma(\beta n + \alpha) 
\int^b_a x(t) ^{RL}I^{\alpha + n\beta}_{a+} y(t)dt 
= \sum^{\infty}_{n=0}a_n \Gamma(\beta n + \alpha)  
\int^b_a y(t) ^{RL}I^{\alpha + n\beta}_{b-} x(t)dt.
\]
Therefore, we obtain that
\[
\int^b_a x(t)^{A}I^{\alpha, \beta}_{a+} y(t) dt 
= \int^b_a y(t)^{A}I^{\alpha, \beta}_{b-} x(t)dt.
\]
The proof is complete.
\end{proof}

\begin{lemma}[Integration by parts formula]
\label{lem:int:parts}
Let $x$ be a continuous function and $y$ a continuously differentiable function. Then, 
\[
\int^b_a x(t) ^{A}_{C}D^{\alpha, \beta}_{a+} y(t)dt 
= \left[y(t)^{\bar{A}}I^{1-\alpha, \beta}_{b-}\right]^b_a 
+ \int^b_a y(t) ^{A}_{RL}D^{\alpha, \beta}_{b-} x(t)dt. 
\]
\end{lemma}

\begin{proof}
By definition, 
\[
\int^b_a x(t) ^{A}_{C}D^{\alpha, \beta}_{a+} y(t)dt 
= \int^b_a x(t) ^{\bar{A}}I^{1-\alpha, \beta}_{a+}y'(t)dt
\]
and, by the duality formula, it follows that 
\[
\int^b_a x(t) ^{\bar{A}}I^{1-\alpha, \beta}_{a+}y'(t)dt 
= \int^b_a y'(t) ^{\bar{A}}I^{1-\alpha, \beta}_{b-}x(t)dt.
\]
Using (standard) integration by parts, we obtain that
\[
\int^b_a y'(t) ^{\bar{A}}I^{1-\alpha, \beta}_{b-}x(t)dt 
= \left[y(t)^{\bar{A}}I^{1-\alpha, \beta}_{b-}\right]^b_a 
- \int^b_a y(t)\frac{d}{dt}\Big( ^{\bar{A}}I^{1-\alpha, \beta}_{b-}x(t)\Big) dt,
\]
which leads to the desired formula.
\end{proof}

\begin{theorem}[Gr\"onwall's inequality]
\label{thm:GI}
Let $\alpha$ be a positive real number and let $a(\cdot), \, g(\cdot)$ and $u(\cdot)$ 
be non-negative continuous functions on $[0, T]$ with $g(\cdot)$ monotonic increasing, 
satisfying $\displaystyle{\underset{t \in [0, T]}\max g(t) < \frac{1}{T^{\alpha}}}$. If 
\begin{equation}
\label{inequality}
u(t)\leqslant a(t)+ g(t)\left({^{A}I^{\alpha, \beta}_{0^{+}}}u\right)(t), 
\end{equation}
then 
\[
u(t) \leqslant a(t)+ \sum_{k=1}^{\infty}g^{k}(t)\Big(  ^{A}I^{k\alpha, \beta}_{0^{+}}a\Big)(t) 
\] 
for any $t\in [0, T]$.
\end{theorem}

\begin{proof}
Because $^{A}I^{\alpha, \beta}_{0^{+}}$ is a non-decreasing operator, we have
\[
\Big(^{A}I^{\alpha, \beta}_{0^{+}}u\Big)(t)
\leqslant {^{A}I^{\alpha, \beta}_{0^{+}}}\Big(a(\cdot) 
+ g(t)\Big(^{A}I^{\alpha, \beta}_{0^{+}}u\Big)\Big)(t)
= \Big(^{A}I^{\alpha, \beta}_{0^{+}}a\Big)(t) 
+ g(t)\Big(^{A}I^{\alpha, \beta}_{0^{+}}u\Big)(t).
\]
Now, using its semi-group property \eqref{semigroup}, we can substitute
the previous inequality into \eqref{inequality}, to obtain 
\[
u(t) \leqslant a(t) + g(t)\Big(^{A}I^{\alpha, \beta}_{0^{+}}a\Big)(t) 
+ g^2(t)\Big(^{A}I^{2\alpha, \beta}_{0^{+}}u\Big)(t).
\]
Repeating this procedure up to $N$ times, we get
\[
u(t) \leqslant a(t) + \sum_{k=1}^{N-1}g^k(t)\Big(
^{A}I^{k\alpha, \beta}_{0^{+}}a\Big)(t) 
+ g^N(t)\Big(^{A}I^{N\alpha, \beta}_{0^{+}}u\Big)(t).
\]
Therefore, when $N\rightarrow \infty$, we have
\[
u(t) \leqslant a(t) + \sum_{k=1}^{\infty}g^k(t)\Big(^{A}
I^{k\alpha, \beta}_{0^{+}}a\Big)(t) 
+ \lim_{N\rightarrow \infty}g^N(t)\Big(^{A}I^{N\alpha, \beta}_{0^{+}}u\Big)(t).
\] 
It remains to show that the series 
\[
\sum_{k=1}^{\infty}g^k(t)\Big(^{A}I^{k\alpha, \beta}_{0^{+}}a\Big)(t)
\]
is convergent and $\underset{{N\rightarrow \infty}}
\lim g^N(t)\Big(^{A}I^{N\alpha, \beta}_{0^{+}}u\Big)(t)=0$, 
to obtain the desired result. Using Definition~\ref{def1} 
of the operator $^{A}I^{k\alpha, \beta}_{0^{+}}$, one has
\[
\sum_{k=1}^{\infty}g^k(t)\Big(^{A}I^{k\alpha, \beta}_{0^{+}}a\Big)(t)
= \sum_{k=1}^{\infty}g^k(t)\int^t_0(t-s)^{k\alpha -1}A\Big((t-s)^{\beta}\Big) a(s)ds.
\]
Next,
\begin{align*}
\left| \sum_{k=1}^{\infty}g^k(t)\Big(^{A}I^{k\alpha, \beta}_{0^{+}}a\Big)(t) \right| 
&\leqslant \mu \sum_{k=1}^{\infty}\left| g^k(t)T^{k\alpha -1} \right|\int^t_0|a(s)|ds\\
& = \mu \sum_{k=1}^{\infty}\left| g^k(t)T^{k\alpha -1} \right|\int^t_0|a(s)|ds\\
& \leqslant \mu \sum_{k=1}^{\infty}\left| M^kT^{k\alpha -1} \right|\int^t_0|a(s)|ds, 
\end{align*}
where $\mu= \underset{|x|< T^{\beta}}\sup A(x)$ and 
$M= \displaystyle{\underset{t \in [0, T]}\max g(t)}$. Hence, we obtain that
\[
\left| \sum_{k=1}^{\infty}g^k(t)\Big(^{A}I^{k\alpha, \beta}_{0^{+}}a\Big)(t) \right| 
\leqslant \mu  T^{-1} 
\sum_{k=1}^{\infty}\left|\left( MT^{\alpha} \right)^{k} \right|\int^t_0|a(s)|ds. 
\]
The series converges providing that $\displaystyle{|M|< \frac{1}{T^{\alpha}}}$. 
Moreover, according to the necessary condition of convergence of an infinite series, 
one has 
$$
\underset{{k\rightarrow \infty}}\lim g^k(t)^{A}I^{N\alpha, \beta}_{0^{+}}=0
$$ 
for all $t\in [0, T]$, which leads to
$\underset{{N\rightarrow \infty}}\lim g^N(t)\Big(^{A}I^{N\alpha, \beta}_{0^{+}}u\Big)(t)=0$. 
\end{proof}


\section{Main Results}
\label{sectionR}

Here we consider a basic optimal control problem, 
which consists to find a piecewise continuous control 
function $u(\cdot) \in PC\left( [a, b]; \mathbb{R}^m\right)$ 
and its corresponding state trajectory $x(\cdot) \in PC^{1}\left( [a, b]; \mathbb{R}^n\right)$, 
solution to problem
\begin{equation}
\label{bp}
\begin{gathered}
\mathcal{J}[x(\cdot), u(\cdot)]
=  \frac{1}{\Gamma(\alpha)A(1)}{^{\bar{A}}
I^{\alpha, \beta}_{a+}\left[L\left(\cdot, x(\cdot), u(\cdot) \right) \right](b)}\longrightarrow \max,\\
^{A}_{C}D^{\alpha, \beta}_{{a}^{+}}x(t)= f\left(t, x(t), u(t)\right), \quad t\in [a, b],\\
x(\cdot) \in PC^{1}, \quad u(\cdot) \in PC,\\
x(a)= x_a \in \mathbb{R}^n,
\end{gathered}
\end{equation}
where, for given $(n, m)\in \mathbb{N}^2$ such that $m\leqslant n $, 
functions $L: [a, b]\times \mathbb{R}^n \times \mathbb{R}^m \rightarrow \mathbb{R}$ 
and $f : [a, b]\times \mathbb{R}^n \times \mathbb{R}^m \rightarrow \mathbb{R}^n$ 
are assumed to be continuously differentiable in all their three arguments: 
$L \in C^1 \left([a, b]\times \mathbb{R}^n \times \mathbb{R}^m ; \mathbb{R}\right)$ 
and $f\in C^1 \left([a, b]\times \mathbb{R}^n \times \mathbb{R}^n ;\mathbb{R}\right)$.

\begin{notation}
We set $w(\cdot):=\displaystyle{\frac{(b-\cdot)^{\alpha-1}
A\left( (b-\cdot)^{\beta} \right)}{\Gamma(\alpha)A(1)}}$, so that
\[
\mathcal{J}[x(\cdot), u(\cdot)]= \int^b_a w(t)L\left(t, x(t), u(t) \right)dt.
\]
\end{notation}


\subsection{Continuity of solutions of control differential equations}

Now, we consider the following control differential equation:
\begin{equation}
\label{eqcontrol}
^{A}_{C}D^{\alpha, \beta}_{{a}^{+}}x(t)
= f\left(t, x(t), u(t)\right), \quad x(a)= x_a, 
\quad t\in [a, b],
\end{equation}
where $x(\cdot) \in PC^{1}\left( [a, b]; \mathbb{R}^n\right)$ 
represents the state trajectory of \eqref{eqcontrol}, 
$u(\cdot) \in PC\left( [a, b]; \mathbb{R}^m\right)$ is the control input, 
and function $f$ is assumed to be Lipschitz-continuous 
with respect to both $x$ and $u$.

\begin{lemma}
\label{cont}
Let us denote by $u^{*}$ a precise control input to \eqref{eqcontrol}, 
and $x^{*}$ its associated state trajectory. Suppose that $u^{\epsilon}$ 
is a control perturbation around the control input $u^{*}$, that is, 
for all $t\in [a,b]$, $u^{\epsilon}(t)= u^{*}(t)+ \epsilon h(t)$, 
where $h(\cdot)\in PC\left( [a, b]; \mathbb{R}^m\right)$ 
is a variation function and $\epsilon \in \mathbb{R}$. Denote 
by $x^{\epsilon}$ its corresponding state trajectory, solution of 
\[
^{A}_{C}D^{\alpha, \beta}_{{a}^{+}}x^{\epsilon}(t)
= f\left(t, x^{\epsilon}(t), u^{\epsilon}(t)\right), 
\quad x^{\epsilon}(a)= x_a.
\]
Then, we have that $x^{\epsilon}$ converges 
to $x^{*}$ when $\epsilon $ tends to zero.
\end{lemma}

\begin{proof}
By definition of the differential operator, we have 
\[
^{A}_{C}D^{\alpha, \beta}_{{a}^{+}}x^{\epsilon}(t)
- ^{A}_{C}D^{\alpha, \beta}_{{a}^{+}}x^{*}(t)
= f\left(t, x^{\epsilon}(t), u^{\epsilon}(t)\right) 
- f\left(t, x^{*}(t), u^{*}(t)\right).
\]
By linearity of the operator and applying its inverse operation,
\[
x^{\epsilon}-x^{*}={^{\bar{A}}I^{\alpha, \beta}_{a^{+}}}\Big( 
f\left(t, x^{\epsilon}(t), u^{\epsilon}(t)\right) 
- f\left(t, x^{*}(t), u^{*}(t)\right) \Big), 
\]
where $\bar{A}_{\Gamma}\cdot A_{\Gamma}=1$. Next, 
\[
\Vert x^{\epsilon}-x^{*}\Vert 
\leqslant {^{\bar{A}}I^{\alpha, \beta}_{a^{+}}}\Big(
\Vert f\left(t, x^{\epsilon}(t), u^{\epsilon}(t)\right) 
- f\left(t, x^{*}(t), u^{*}(t)\right)\Vert \Big)
\]
and, by the Lipschitz-property of $f$ 
and the non-decreasing property of $^{A}I^{\alpha, \beta}_{a^{+}}$,
\begin{align*}
\Vert x^{\epsilon}-x^{*}\Vert 
&\leqslant {^{\bar{A}}I^{\alpha, \beta}_{a^{+}}}\Big( L_1\Vert x^{\epsilon}
-x^{*}\Vert \Big)+ ^{\bar{A}}I^{\alpha, \beta}_{a^{+}}\Big( L_2|\epsilon| \Vert h(t)\Vert \Big)\\
&= L_2|\epsilon|^{\bar{A}}I^{\alpha, \beta}_{a^{+}}\Big( 
\Vert h(t)\Vert \Big) + L_1 ^{\bar{A}}
I^{\alpha, \beta}_{a^{+}}\Big( \Vert x^{\epsilon}-x^{*}\Vert \Big).
\end{align*}
Now, applying Gr\"onwall's inequality (Theorem~\ref{inequality}), it follows that
\begin{align*}
\Vert x^{\epsilon}-x^{*}\Vert 
&\leqslant L_2|\epsilon|^{\bar{A}}I^{\alpha, \beta}_{a^{+}}\Big( \Vert h(t)\Vert \Big) 
+ \sum_{k=1}^{\infty}L_1^k\Big[^{\bar{A}}I^{k\alpha, \beta}_{a^{+}}\left( 
L_2|\epsilon|^{\bar{A}}I^{\alpha, \beta}_{a^{+}}\Big( \Vert h(t)\Vert \Big) \right) \Big]\\
&= |\epsilon| L_2 \left[ ^{\bar{A}}I^{\alpha, \beta}_{a^{+}}\Big( \Vert h(t)\Vert \Big) 
+ \sum_{k=1}^{\infty}L_1^k\left(^{\bar{A}}I^{(k+1)\alpha, \beta}_{a^{+}}\Big( 
\Vert h(t)\Vert \Big) \right)\right].
\end{align*}
Finally, due to Theorem~\ref{inequality}, 
if $\displaystyle{L_1 < \frac{1}{T^{\alpha}}}$, then the series 
\[
\sum_{k=1}^{\infty}L_1^k\left(^{\bar{A}}I^{(k+1)\alpha, \beta}_{a^{+}}\Big(
\Vert h(t)\Vert \Big) \right)
\]
converges and, by taking the limit when $\epsilon \rightarrow 0$, 
we obtain the desired result, that is,  
$x^{\epsilon}(t) \rightarrow x^{*}(t)$ for all $t \in [a,b]$.
\end{proof}


\subsection{Differentiability of solutions}

The following result is useful for the proof of our 
necessary optimality condition in Section~\ref{sec:PMP}.

\begin{lemma}[Differentiability of perturbed trajectories]
\label{different}
There exists a function $\eta$ defined on $[a,b]$ such that 
\[
x^{\epsilon}(t)= x^{*}(t) + \epsilon \eta(t) + o(\epsilon).
\]
\end{lemma}

\begin{proof}
Since $f \in C^1$, we have for any fixed index $j$ that
\begin{multline*}
f_j(t, x^{\epsilon}, u^{\epsilon})
= f_j(t, x^{*}, u^{*}) + \sum_{i=1}^n(x^{\epsilon}_i
-x^{*}_i)\frac{\partial f_j(t, x^{*}_i, u^{*})}{\partial x}\\
+ \sum_{i=1}^m(u^{\epsilon}_i-u^{*}_i)
\frac{\partial f_j(t, x^{*}, u^{*}_i)}{\partial u}
+ o(|x^{\epsilon}_i - x^{*}_i|,|u^{\epsilon}_i-u^{*}_i|).
\end{multline*}
Observe that $u^{\epsilon}_i-u^{*}_i= \epsilon h_i(t)$ 
and $u^{\epsilon}_i \rightarrow u^{*}_i$ when $\epsilon \rightarrow 0$ and, 
by Theorem~\ref{cont}, we have $x^{\epsilon}_i \rightarrow x^{*}_i$ 
when $\epsilon \rightarrow 0$. Thus, the residue term can be expressed 
in terms of $\epsilon$ only, that is, the residue is $o(\epsilon)$. 
Therefore, for all indexes $j\in \{1, \ldots, n\}$, we have the vector expression
\[
f(t, x^{\epsilon}, u^{\epsilon})
= f(t, x^{*}, u^{*}) + (x^{\epsilon}-x^{*}) J_x(t, x^{*}, u^{*}) 
+ \epsilon h(t) J_u(t, x^{*}, u^{*}) + o(\epsilon),
\]
where $J_x(t, x^{*}, u^{*}) $ and $J_u(t, x^{*}, u^{*}) $ 
are the Jacobian matrices of $f$ with respect to $x$ and $u$, respectively, 
and evaluated at $(t, x^{*}, u^{*})$, that is, 
\begin{equation}
J_x(t, x^{*}, u^{*})
= \begin{pmatrix}
\frac{\partial f_j}{\partial x_i}
\end{pmatrix}_{1\leqslant i, j\leqslant n}
\quad \text{ and }
J_u(t, x^{*}, u^{*})= 
\begin{pmatrix}
\frac{\partial f_j}{\partial u_i}
\end{pmatrix}
_{1\leqslant i\leqslant m, \  1\leqslant j\leqslant n}.
\end{equation}  
Next, we have 
\[
^{A}_{C}D^{\alpha, \beta}_{{a}^{+}}x^{\epsilon}
= ^{A}_{C}D^{\alpha, \beta}_{{a}^{+}}x^{*} 
+ (x^{\epsilon}-x^{*})\cdot J_x(t, x^{*}, u^{*}) 
+ \epsilon h(t)\cdot J_u(t, x^{*}, u^{*}) + o(\epsilon)
\]
and this leads to 
\[
\lim_{\epsilon \rightarrow 0} \left[
\frac{^{A}_{C}D^{\alpha, \beta}_{{a}^{+}}( x^{\epsilon}-x^{*})}{\epsilon} 
-\frac{( x^{\epsilon}-x^{*})}{\epsilon}\cdot J_x(t, x^{*}, u^{*}) 
- h(t)\cdot J_u(t, x^{*}, u^{*}) \right]=0,
\]
that is,
\[
^{A}_{C}D^{\alpha, \beta}_{{a}^{+}} \left(\lim_{\epsilon \rightarrow 0} 
\frac{  x^{\epsilon}-x^{*}}{\epsilon} \right) 
= \lim_{\epsilon \rightarrow 0}  
\frac{( x^{\epsilon}-x^{*})}{\epsilon}\cdot J_x(t, x^{*}, u^{*})  
+ h(t)\cdot J_u(t, x^{*}, u^{*}). 
\]
It remains to prove the existence of the limit
$\displaystyle{\lim_{\epsilon \rightarrow 0} 
\frac{x^{\epsilon}-x^{*}}{\epsilon}}$. 
For this purpose, we set $\eta := \displaystyle{\lim_{\epsilon \rightarrow 0} 
\frac{x^{\epsilon}-x^{*}}{\epsilon}}$. It is easy to see that the limit exists 
as solution of the following system of fractional differential equations:
\begin{equation*}
\begin{cases}
^{A}_{C}D^{\alpha, \beta}_{{a}^{+}} \eta(t)
= \eta(t)\cdot J_x(t, x^{*}, u^{*})  
+h(t)\cdot J_u(t, x^{*}, u^{*}) ,\\[3mm]
\eta(a)= 0.
\end{cases}
\end{equation*}
This ends the proof.
\end{proof}


\subsection{Pontryagin's maximum principle} 
\label{sec:PMP}

The following result is a necessary optimality condition 
of Pontryagin type for problem \eqref{bp}.

\begin{theorem}[Pontryagin Maximum Principle for \eqref{bp}]
\label{theo}
If $(x^{*}(\cdot), u^{*}(\cdot))$ is an optimal pair for \eqref{bp}, 
then there exists $\lambda_0 \in \{0,1\}$ and 
$\lambda(\cdot) \in PC^1\left( [a, b]; \mathbb{R}^n\right)$, 
called the adjoint variables, such that the following conditions 
hold for all $t$ in the interval $[a, b]$:
\begin{itemize}
\item the nontriviality condition
\begin{equation}\label{trivial}
\left(\lambda_0, \lambda\right)\neq (0,0);
\end{equation}
\item the optimality condition
\begin{equation}\label{opt}
\nabla_{u} H\left( (t,x^{*}(t), u^{*}(t), \lambda_0, \lambda(t)\right)=0;
\end{equation}
\item the adjoint system
\begin{equation}
\label{adj}
^{A}_{RL}D^{\alpha, \beta}_{b^{-}} \lambda(t)
=\nabla_{x} H\left( (t,x^{*}(t), u^{*}(t), \lambda_0, \lambda(t)\right);
\end{equation}
\item the transversality condition
\begin{equation}
\label{trans}
^{\bar{A}}I^{1-\alpha, \beta}_{b^{-}}\lambda(b)=0;
\end{equation}
\end{itemize}
where function $H$, defined by
\[
H\left(t, x, u, \lambda_0,\lambda \right)= \lambda_0 w(t)L(t, x, u) 
+ \sum_{j=1}^n \lambda_j f_j(t, x, u),
\]
is called the Hamiltonian.
\end{theorem}

\begin{proof}
Let $(x^{*}(\cdot), u^{*}(\cdot))$ be solution of the problem, 
$h(\cdot) \in PC\left( [a, b]; \mathbb{R}^m\right)$ be a variation, 
and $\epsilon $ a real constant. Define 
$u^{\epsilon}(t)= u^{*}(t)+ \epsilon h(t)$, so that 
$u^{\epsilon}\in PC\left( [a, b]; \mathbb{R}^m\right)$. 
Let $x^{\epsilon}$ be the corresponding trajectory to the control $u^{\epsilon}$, 
meaning it is the state solution to the following system: 
\begin{equation}
\label{equatepsi}
^{A}_{C}D^{\alpha, \beta}_{{a}^{+}} x^{\epsilon}(t)
= f\left(t, x^{\epsilon}(t), u^{\epsilon}(t)\right), 
\quad x^{\epsilon}(a)= x_a.
\end{equation}
Note that $u^{\epsilon}(t) \rightarrow u^{*}(t)$ for all $t\in [a, b]$ 
whenever $\epsilon \rightarrow 0$. Further, the first derivative of 
$u^{\epsilon}$ with respect to $\epsilon$ at $\epsilon = 0$ can be obtained as
\begin{equation}
\label{eqpartialu}
\displaystyle{\frac{\partial 
u^{\epsilon}(t)}{\partial \epsilon}\Bigr|_{\epsilon=0}  = h(t)}.
\end{equation} 
Similarly, by Lemma~\ref{cont}, it follows that $x^{\epsilon}(t)\rightarrow x^{*}(t)$, 
for each fixed $t$, as $ \epsilon \rightarrow 0$. Also, from Lemma~\ref{different}, 
the first derivative of $x^{\epsilon}$ with respect to $\epsilon$ at $\epsilon = 0$, 
\begin{equation}
\label{eqpartialx}
\displaystyle{\frac{\partial x^{\epsilon}(t)}{\partial 
\epsilon}\Bigr|_{\epsilon=0}},
\end{equation}
exists for each $t$. The objective functional at $(x^{\epsilon}, u^{\epsilon})$ is 
\[
\mathcal{J}(x^{\epsilon}, u^{\epsilon})
= \int^b_a Lw(t)\left(t, x^{\epsilon}(t), u^{\epsilon}(t) \right)dt.
\]
This functional can be extended to handle abnormal 
multipliers in the following way:
\[
\mathcal{J}_{\lambda_0}(x^{\epsilon}, u^{\epsilon})
= \lambda_0 \mathcal{J}(x^{\epsilon}, u^{\epsilon})
= \int^b_a \lambda_0 w(t) 
L\left(t, x^{\epsilon}(t), u^{\epsilon}(t) \right)dt,
\]
where $\lambda_0 \in \{ 0, 1\}$ and the case $\lambda_0 =0$ 
is known as the abnormal case. Next, we introduce the adjoint 
vector function $\lambda$. Let $\lambda (\cdot)$ be in 
$PC^{1}\left( [a, b]; \mathbb{R}^n\right)$, to be determined. 
By the integration by parts formula, we have, for any fix index $j$, that
\[
\int^b_a \lambda_j(t) ^{A}_{C}D^{\alpha, \beta}_{{a}^{+}}x^{\epsilon}_j(t)dt 
= \left[ x^{\epsilon}_j(t) ^{\bar{A}}I^{1-\alpha, \beta}_{b^{-}}\lambda_j(t) 
\right]^b_a + \int^b_ax^{\epsilon}_j(t) ^{A}_{RL}D^{\alpha, \beta}_{b^{-}} \lambda_j (t)dt
\]
and, summing up for $j=1, \ldots, n$, we get 
the following expression in inner product form:
\begin{multline*}
\int^b_a \lambda(t)\cdot ^{A}_{C}D^{\alpha, \beta}_{{a}^{+}}
x^{\epsilon}(t)dt - \int^b_ax^{\epsilon}(t)
\cdot ^{A}_{RL}D^{\alpha, \beta}_{b^{-}} \lambda (t)dt 
-  x^{\epsilon}(b)\cdot ^{\bar{A}}I^{1-\alpha, \beta}_{b^{-}}\lambda(b)\\
+ x^{\epsilon}(a)\cdot ^{\bar{A}}I^{1-\alpha, \beta}_{b^{-}}\lambda(a) =0.
\end{multline*}
Adding this zero to the expression 
$\mathcal{J}_{\lambda_0}(x^{\epsilon}, u^{\epsilon})$ gives 
\begin{gather*}
\phi (\epsilon)= \mathcal{J}_{\lambda_0}(x^{\epsilon}, u^{\epsilon})
= \int^b_a \left[\lambda_0 w(t)L\left( t, x^{\epsilon}(t), u^{\epsilon}(t)\right) 
+ \lambda(t)\cdot ^{A}_{C}D^{\alpha, \beta}_{{a}^{+}}x^{\epsilon}(t)\right. \\
\left.  - x^{\epsilon}(t)\cdot ^{A}_{RL}D^{\alpha, \beta}_{b^{-}} \lambda (t) \right]dt 
- x^{\epsilon}(b)\cdot ^{\bar{A}}I^{1-\alpha, \beta}_{b^{-}}\lambda(b)
+ x^{\epsilon}(a)\cdot ^{\bar{A}}I^{1-\alpha, \beta}_{b^{-}}\lambda(a), 
\end{gather*}
which by \eqref{equatepsi} is equivalent to
\begin{gather*}
\phi (\epsilon)= \mathcal{J}_{\lambda_0}(x^{\epsilon}, u^{\epsilon})
= \int^b_a \left[\lambda_0 w(t)L\left( t, x^{\epsilon}(t), u^{\epsilon}(t)\right) 
+ \lambda(t)\cdot f\left(t, x^{\epsilon}(t), u^{\epsilon}(t) \right)\right. \\
\left. - x^{\epsilon}(t)\cdot ^{A}_{RL}D^{\alpha, \beta}_{b^{-}} \lambda (t) \right]dt 
- x^{\epsilon}(b)\cdot ^{\bar{A}}I^{1-\alpha, \beta}_{b^{-}}\lambda(b)
+ x^{\epsilon}(a)\cdot ^{\bar{A}}I^{1-\alpha, \beta}_{b^{-}}\lambda(a). 
\end{gather*}
The overall optimization problem is now reduced to the study of function $\phi$ and, 
for this purpose, we must have $\phi$ non identically zero ($\phi \neq 0$). 
To ensure this, it is sufficient to consider that $\left(\lambda_0, \lambda(t)\right)\neq (0,0)$, 
meaning that both multipliers can not vanish simultaneously.
The maximum of $\mathcal{J}_{\lambda_0}$ occurs at $(x^{*}, u^{*})= (x^0, u^0)$, 
so the derivative of $\phi(\epsilon)$ with respect to $\epsilon $ 
at $\epsilon=0 $ must vanish, that is,
\begin{align*}
0&= \phi'(0)= \frac{d }{d \epsilon } J_{\lambda_0}(x^{\epsilon}, u^{\epsilon})|_{\epsilon=0}\\
&= \int^b_a \left[ \lambda_0 w(t)\left(\sum_{i=1}^n\frac{\partial L}{\partial x_i}
\frac{\partial x^{\epsilon}_i(t)}{\partial \epsilon}\Bigr|_{\epsilon=0} 
+ \sum_{i=1}^m \frac{\partial L}{\partial u_i}\frac{\partial 
u^{\epsilon}_i(t)}{\partial \epsilon}\Bigr|_{\epsilon=0} \right)\right. \\
& \quad \left. +  \lambda_1 (t)\left( \sum_{i=1}^n\frac{\partial f_1}{\partial x_i}
\frac{\partial x^{\epsilon}_i(t)}{\partial \epsilon}\Bigr|_{\epsilon=0} 
+ \sum_{i=1}^m\frac{\partial f_1}{\partial u_i}\frac{\partial 
u^{\epsilon}_i(t)}{\partial \epsilon}\Bigr|_{\epsilon=0}\right)\right.\\  
&\quad \left. + \cdots + \lambda_n (t) \sum_{i=1}^n\frac{\partial f_n}{\partial x_i}
\frac{\partial x^{\epsilon}_i(t)}{\partial \epsilon}\Bigr|_{\epsilon=0}\right. \\
& \quad \left.  + \lambda_n (t) \sum_{i=1}^m\frac{\partial f_n}{\partial u_i}
\frac{\partial u^{\epsilon}_i(t)}{\partial \epsilon}\Bigr|_{\epsilon=0} 
-^{A}_{RL}D^{\alpha, \beta}_{b^{-}} \lambda_1 (t)\frac{\partial 
x^{\epsilon}_1(t)}{\partial \epsilon}\Bigr|_{\epsilon=0}\right.\\ 
&\quad \left. - \cdots   -^{A}_{RL}D^{\alpha, \beta}_{b^{-}} \lambda_n (t)
\frac{\partial x^{\epsilon}_n(t)}{\partial \epsilon}\Bigr|_{\epsilon=0} \right]dt \\
&\quad - {^{\bar{A}}I^{1-\alpha, \beta}_{b^{-}}}\lambda_1(b) 
\frac{\partial x^{\epsilon}_1(b)}{\partial \epsilon}\Bigr|_{\epsilon=0}
-\cdots  -^{\bar{A}}I^{1-\alpha, \beta}_{b^{-}}\lambda_1(b) 
\frac{\partial x^{\epsilon}_n(b)}{\partial \epsilon}\Bigr|_{\epsilon=0},
\end{align*}
where the partial derivatives of $L$ and $f=(f_1, \cdots, f_n)$ with respect 
to $x$ and $u$, $x=(x_1, \cdots, x_n)$ and $u=(u_1, \cdots, u_m)$, are evaluated at 
$ \left( t, x^{*}(t), u^{*}(t) \right)$. Thus, using \eqref{eqpartialu} 
and \eqref{eqpartialx}, and rearranging the terms, we obtain that
\begin{multline*}
\int^b_a \left[\left( \lambda_0 w(t)\nabla_x L 
+ \lambda(t) J_x-^{A}_{RL}D^{\alpha, \beta}_{b^{-}} \lambda(t)\right)
\cdot \displaystyle{ \frac{\partial x^{\epsilon}(t)}{\partial 
\epsilon}\Bigr|_{\epsilon=0} } \right. \\
\left. + \left( \lambda_0 w(t)\nabla_u L + \lambda(t)J_u\right)
\cdot \displaystyle{\frac{\partial u^{\epsilon}(t)}{\partial 
\epsilon}\Bigr|_{\epsilon=0}}  \right]dt -^{\bar{A}}I^{1-\alpha, \beta}_{b^{-}}
\lambda(b)\cdot \displaystyle{ \frac{\partial 
x^{\epsilon}(t)}{\partial \epsilon}\Bigr|_{\epsilon=0} } = 0,
\end{multline*}
where $J_x$ and $J_u$ are the Jacobian matrices of $f$, respectively with respect to $x$ and $u$  
and evaluated at $(t, x^{*}, u^{*})$, that is, 
\begin{equation}
J_x= 
\begin{pmatrix}
\frac{\partial f_j}{\partial x_i}
\end{pmatrix}_{1\leqslant i, j\leqslant n}
\quad \text{and} \quad
J_u= \begin{pmatrix}
\frac{\partial f_j}{\partial u_i}
\end{pmatrix}_{1\leqslant i\leqslant m, \, \,  1\leqslant j\leqslant n}.
\end{equation}
Setting $H= \lambda_0 w(t)L + \lambda \cdot f$, it follows that
\begin{multline*}
\int^b_a \left[ \Big(\nabla_x H -  ^{A}_{RL}D^{\alpha, \beta}_{b^{-}} 
\lambda (t)\Big)\cdot \displaystyle{ \frac{\partial x^{\epsilon}(t)}{\partial 
\epsilon}\Bigr|_{\epsilon=0} }  + \nabla_u H \cdot h(t)\right]dt \\
- ^{\bar{A}}I^{1-\alpha, \beta}_{b^{-}}\lambda(b)
\frac{\partial x^{\epsilon}(b)}{\partial \epsilon}\Bigr|_{\epsilon=0} =0,
\end{multline*}
where the partial derivatives of $H$ are evaluated at
$\left( t, x^{*}(t), u^{*}(t), \lambda_0, \lambda(t) \right)$. Now, choosing
\[
^{A}_{RL}D^{\alpha, \beta}_{b^{-}} \lambda (t) = \nabla_x H, 
\quad \text{ with } ^{\bar{A}}I^{1-\alpha, \beta}_{b^{-}}\lambda(b)=0,
\]
that is, given the adjoint equation \eqref{adj} 
and the transversality condition \eqref{trans}, one obtains
\[
\int^b_a \nabla_u H \cdot h(t)=0
\]
and, by the fundamental lemma of the  calculus of variations \cite{MR500859}, 
we have the optimality condition \eqref{opt}:
\[
\nabla_u H\left( t, x^{*}(t), u^{*}(t), \lambda_0, \lambda(t) \right) =0.
\]
This concludes the proof.
\end{proof}

\begin{example}
Let us consider the following problem:
\begin{equation}
\label{exple}
\begin{gathered}
\int^2_0 -w(t)\left[\parallel x(t)-(t^2, e^{1-t}) \parallel^2 
+ \parallel u(t)-(t^2e^{-t}, -t^6) \parallel^2 \right]dt \longrightarrow \max,\\
\begin{cases}
^{A}_{C}D^{\frac{1}{3}, \pi}_{{a}^{+}}x_1(t)= u_1(t),\\
^{A}_{C}D^{\frac{1}{3}, \pi}_{{a}^{+}}x_2(t)= u_2^2(t) + 2t^6u_2(t),
\end{cases}\\
x_1(0)=0, \quad x_2(0)= e,
\end{gathered}
\end{equation}
where $w(t)= \displaystyle{\frac{(2-t)^{\frac{1}{3}-1}
A\left( (2-t)^{\pi} \right)}{\Gamma(\alpha)A(1)}}$. 
In order to apply Theorem~\ref{theo}, let us define 
the normal Hamiltonian: 
\begin{multline*}
H(t, x, u, \lambda)= -w(t)\left[\parallel x(t)
-(t^2, e^{1-t}) \parallel^2 + \parallel u(t)
-(t^2e^{-t}, -t^6) \parallel^2 \right] \\
+ \lambda_1 u_1 + \lambda_2 u^2_2 + 2\lambda_2 t^6 u_2.
\end{multline*}
We have:
\begin{itemize}
\item by the optimality condition,
\begin{equation}
\label{exopt}
\begin{cases}
\lambda_1(t)=-2\left( u_1(t)-t^2e^{-t}\right),\\[3mm] 
2\left( u_2(t)+t^6\right)\left(1+\lambda_2(t) \right)=0;
\end{cases}
\end{equation}
\item by the adjoint system,
\begin{equation}
\label{exadjt}
\begin{cases}
^{A}_{C}D^{\frac{1}{3}, \pi}_{{0}^{+}}= 2\left(x_1(t)-t^2 \right),\\[3mm]
^{A}_{C}D^{\frac{1}{3}, \pi}_{{2}^{+}} = 2\left( x_2(t)-e^{1-t}\right);
\end{cases}
\end{equation}
\item by the transversality condition,
\begin{equation}
\label{extrans}
\begin{cases}
^{\bar{A}}I^{\frac{1}{3}, \pi}_{2-}\lambda_1(2)=0,\\
^{\bar{A}}I^{\frac{1}{3}, \pi}_{2-}\lambda_2(2)=0.
\end{cases}
\end{equation}
\end{itemize}
Therefore, we observe that equations \eqref{exopt}, \eqref{exadjt} 
and \eqref{extrans} are satisfied by the following pair of functions: 
$x(t)= (t^2, e^{1-t})$ and $u(t)= (t^2e^{-t}, -t^6)$. These are
the Pontryagin extremals of the problem.
\end{example}


\subsection{Calculus of variations}

The problem of the calculus of variations involving fractional operators 
with a general analytic kernel consists to find a piecewise 
continuously differentiable curve $x$ solution of 
\begin{equation}
\label{cv}
\begin{gathered}
\mathcal{J}[x(\cdot)]
= \frac{1}{\Gamma(\alpha)A(1)}{^{\bar{A}}I^{\alpha, 
\beta}_{a+}\left[L\left(\cdot, x(\cdot), 
^{A}_{C}D^{\alpha, \beta}_{{a}^{+}}x(\cdot) \right) \right](b)} \longrightarrow \max,\\
\textup{ subject to \qquad  \qquad}\\
x(a)= x_a, \quad x(b)=x_b,
\end{gathered}
\end{equation}
where function $L$ is piecewise continuously differentiable, that is, 
\[
L \in PC^{1}\left( [a, b]\times \mathbb{R}^n \times \mathbb{R}^n; \mathbb{R}\right).
\]
The aforementioned variational problem \eqref{cv} is 
a special case of our optimal control problem \eqref{bp}.

\begin{corollary} 
\label{cor:EL:fund:CoV}
If $x^{*}$ is solution to problem \eqref{cv}, 
then it satisfies the the following Euler--Lagrange equation:
\begin{multline}
\label{euler}
^{A}_{RL}D^{\alpha, \beta}_{b^{-}} \left[w(t)\nabla_{u} 
L\left( (t,x^{*}(t), \, \,^{A}_{C}D^{\alpha, \beta}_{a^{+}} x^{*}(t)\right)\right]\\ 
+ w(t)\nabla_{x} L\left( (t,x^{*}(t), \, \, ^{A}_{C}D^{\alpha, \beta}_{a^{+}} x^{*}(t)\right)=0.
\end{multline}
\end{corollary}

\begin{proof}
It is easy to see that the optimal control problem \eqref{bp} coincides,
in the particular case when $^{A}_{C}D^{\alpha, \beta}_{{a}^{+}}x(t) = u(t)$,
with the  problem of the calculus of variations defined in \eqref{cv}. 
Next, we define the Hamiltonian function 
\[
H\left(t, x, u, \lambda_0, \lambda \right)= \lambda_0 w(t) L\left(t, x, u \right) + \lambda\cdot u
\]
and, by application of Theorem~\ref{theo}, we have:
\begin{itemize}
\item from optimality condition \eqref{opt},
\end{itemize}
\begin{equation}
\label{cvopt}
\lambda(t)= -\lambda_0 w(t)\nabla_{u} L\left( (t,x^{*}(t), u^{*}(t)\right).
\end{equation}
It follows that $\lambda_0=0$ implies $\lambda(t)\equiv 0$, 
which is not a possibility by the nontriviality condition \eqref{trivial}. 
Therefore, $\lambda_0=1$.
\item  Following the adjoint system \eqref{adj}, we have 
\begin{equation}
\label{cvadj}
^{A}_{RL}D^{\alpha, \beta}_{b^{-}} \lambda(t)
= w(t)\nabla_{x} L\left( (t,x^{*}(t), u^{*}(t)\right).
\end{equation}
Combining \eqref{cvopt} and \eqref{cvadj}, we obtain the Euler--Lagrange equation \eqref{euler}.
\end{proof}


\subsubsection*{Isoperimetric problems}

An important class of variational problems are the 
fractional isoperimetric problems \cite{MR2921903,MR3826660}.
The isoperimetric problem involving fractional operators 
with a general analytic kernel consists to find a piecewise 
continuously differentiable curve $x$ solution of 
\begin{equation}
\label{ip}
\begin{gathered}
\mathcal{J}[x(\cdot)]=\frac{1}{\Gamma(\alpha)A(1)}{^{\bar{A}}
I^{\alpha, \beta}_{a+}\left[L\left(\cdot, x(\cdot), ^{A}_{C}
D^{\alpha, \beta}_{{a}^{+}}x(\cdot) \right) \right](b)}\longrightarrow \max, \\
\textup{ subject to }\\
^{\bar{A}}I^{\alpha, \beta}_{a+} y\left(t, x(t), 
\, \, ^{A}_{C}D^{\alpha, \beta}_{{a}^{+}}x(t)\right)= l,\\
x(a)= x_a, \quad x(b)=x_b.
\end{gathered}
\end{equation}

The isoperimetric problem \eqref{ip} is also a particular case 
of our optimal control problem \eqref{bp}.

\begin{corollary}
\label{cor:EL:iso}
If $x^{*}$ is solution to problem \eqref{ip}, 
then it satisfies the the following fractional differential equation:
\[
^{A}_{RL}D^{\alpha, \beta}_{b^{-}} \left[\nabla_{u} \left( w(t)\tilde{L} 
+ \lambda \tilde{y}\right)\right] 
+ \nabla_{x} \left( w(t)\tilde{L} + \lambda \tilde{y}\right)=0,
\]
where $\tilde{L}= L\left(t,x^{*}(t), \, \,^{A}_{C}D^{\alpha, \beta}_{a^{+}} x^{*}(t)\right)$,  
$\tilde{y}= y\left( t,x^{*}(t), \, \,^{A}_{C}D^{\alpha, \beta}_{a^{+}} x^{*}(t)\right)$ 
and $\lambda$ is a nonzero real constant.
\end{corollary}

\begin{proof}
We rewrite the isoperimetric problem \eqref{ip} in the following way:
\begin{gather*}
\mathcal{J}[x(\cdot), u(\cdot)]
=\frac{1}{\Gamma(\alpha)A(1)}{^{\bar{A}}I^{\alpha, \beta}_{a+}\left[
L\left(\cdot, x(\cdot), u(\cdot) \right) \right](b)}\longrightarrow \max,\\
\textup{ subject to \qquad  \qquad}\\
\begin{cases}
^{A}_{C}D^{\alpha, \beta}_{{a}^{+}}x(t) = u(t), \\
^{A}_{C}D^{\alpha, \beta}_{{a}^{+}}{z}(t)= y\left(t, x(t), u(t)\right);\\
\end{cases}\\
x(a)=x_a, \quad x(b)=x_b,\\
z(a)=0, \quad z(b)=l.
\end{gather*}
Hence, it is easy to see that the isoperimetric problem \eqref{ip}
is a particular case of problem \eqref{bp}. 
The Hamiltonian function is here given by
\[
H\left(t, x, u, \lambda_0, \lambda \right)
= \lambda_0 w(t) L\left(t, x, u \right) + \mu \cdot u+\lambda y
\]
and, by application of Theorem~\ref{theo}, we have:
\begin{itemize}
\item from the optimality condition \eqref{opt}, that
\end{itemize}
\begin{equation}
\label{ipopt}
\mu= -\lambda_0 w(t)\nabla_{u} L\left( t,x^{*}(t), u^{*}(t)\right) 
- \lambda \nabla_ y\left( t,x^{*}(t), u^{*}(t)\right);
\end{equation}
from the adjoint system \eqref{adj}, that
\begin{equation}
\label{ipadj}
^{A}_{RL}D^{\alpha, \beta}_{b^{-}} 
\begin{pmatrix} \mu \\
\lambda
\end{pmatrix}  
= 
\begin{pmatrix}
\lambda_0 w(t)\nabla_{x} L\left( t,x^{*}(t), u^{*}(t)\right) 
+ \lambda \nabla_x y\left( t,x^{*}(t), u^{*}(t)\right)\\
0
\end{pmatrix}.
\end{equation}
Now, following the non-triviality condition \eqref{trivial}, 
if $\lambda_0=0$, then 
$$
\mu= - \lambda \nabla_ y\left( t,x^{*}(t), u^{*}(t)\right)
$$ 
and 
$$ 
^{A}_{RL}D^{\alpha, \beta}_{b^{-}} \lambda =0. 
$$
As a consequence, $\lambda $ must be a non-zero constant in order to have 
$$
(\lambda_0, \mu, \lambda)\neq (0, 0, 0). 
$$
Therefore, combining \eqref{ipopt} 
and \eqref{ipadj}, we obtain that
\[
^{A}_{RL}D^{\alpha, \beta}_{b^{-}} \left[\nabla_{u} \left( w(t)\tilde{L} 
+ \lambda \tilde{y}\right)\right] 
+ \nabla_{x} \left( w(t)\tilde{L} + \lambda \tilde{y}\right)=0.
\]
The proof is complete.
\end{proof}


\section*{Acknowledgments}

This research was funded by the Portuguese Foundation for Science and Technology (FCT) 
through CIDMA, grant number UIDB/04106/2020. The first author was supported by FCT 
through the PhD fellowship PD/BD/150273/2019.



\end{document}